\documentclass[]{article}

\usepackage[utf8]{inputenc}
\usepackage[T1]{fontenc}
\usepackage{lmodern}

\usepackage[a4paper]{geometry}
\usepackage{color}
\usepackage{amsmath,amsfonts,amssymb,amscd,latexsym,amsthm,mathtools} 
\usepackage{bbm} 
\usepackage{mathrsfs}
\usepackage{subcaption}
\usepackage{graphicx}
\usepackage{trfsigns}
\usepackage{algorithmic}
\usepackage{algorithm}

\newcommand*{\dt}[1]{{\dot{#1}}}

\newcommand{\Oint}{{\Omega_{\mathrm{int}}}}
\newcommand{\Oout}{{\Omega_{\mathrm{out}}}}
\newcommand{\Om}{{\Omega_{\mathrm{m}}}}
\newcommand{\velocity}{{b}}

\newtheorem{mytheorem}{Theorem}[section]
\newtheorem{mydefinition}[mytheorem]{Definition}
\newtheorem{myremark}[mytheorem]{Remark}

\numberwithin{equation}{section}

\begin{document}
\title{A shape optimization algorithm for interface identification allowing topological changes}
\author{Martin Siebenborn\thanks{Universit\"at Trier, D-54286 Trier, Germany, Email: siebenborn@uni-trier.de}
}
\date{}
\maketitle

\begin{abstract}
\noindent
In this work we investigate a combination of classical PDE constrained optimization methods and a rounding strategy based on shape optimization for the identification of interfaces.
The goal is to identify radioactive regions in a groundwater flow represented by a control that is either active or inactive.
We use a relaxation of the binary problem on a coarse grid as initial guess for the shape optimization with higher resolution.
The result is a computationally cheap method that does not have to perform large shape deformations.
We demonstrate that our algorithm is moreover able to change the topology of the initial guess.
\end{abstract}

\section{Introduction}
\label{sec:introduction}
Parameter estimation and optimal control for problems modeled by partial differential equations usually aims for optimizers which are continuous variables.
Yet, there are many applications, where it is known, that a parameter only takes a discrete number of values.
This could be for instance a spatially distributed binary variable representing a compound of two different materials.
Common problems in this field are ranging from electrical impedance tomography to tolopogy optimization of elastic materials.
See for instance \cite{cheney1999electrical,hintermuller2008electrical,allaire2004structural}.
In this work we focus on groundwater flows through regions that are potentially contaminated with radioactive substances.
The aim is to recover the locations of the radioactive sources based on measurements of the contaminated groundwater.
In applications of these kinds the variable is not the value of the parameter itself but the shape of the interfaces separating the different, discrete states.
This can significantly reduce the degrees of freedom of the optimization problem and thereby enable high spatial resolutions.
PDE-constraint shape optimization has turned out to be very effective for a wide range of applications \cite{berggren2016largescale,mohammadi2001applied,jameson2003aerodynamic,giles2000introduction,allaire2012shape,gangl2015shape}.
Yet, these methods require a good initial guess of the number of shapes and their approximate location, which is usually a strong assumption.
Basically, there are two possible ways to resolve the shapes of the desired interfaces.
It can either be resolved by edges in a finite element mesh or implicitly described by level-set functions.
In both cases an optimization algorithm has to move the mesh or the level-set function along descent directions.
This leads to numerical instabilities for larger deformations, which is the case for an inaccurate initial guess.
One has to deal with a decreasing mesh quality after each optimization step.
Especially larger translations of shapes are problematic.
This is due to the fact that a shape derivative mainly contains information for deformations which are normal to the boundary.
Deformations in tangent direction do not affect the objective function.
Approaches to overcome this issue are presented for instance in \cite{schulz2016computational,schulz2016efficient}.

A common problem with shape representations by level-sets are increasing slopes during the transport of the function.
This requires a so called reinitialization of the level-set after several optimization steps \cite{allaire2004structural}.
Therefore, a level-set has to be determined which represents a given shape.
This is computationally expensive and does not suit for black-box solvers.
For instance in \cite{laurain2016distributed} the combination of level-set shape optimization and volume formulation of shape derivatives is investigated for tomography problems.

Another possible approach, without resolving shapes, is to relax the optimal control problem to continuous parameter values and apply a rounding strategy to the minimizer (see for instance \cite{hante2013relaxation}).

In this work we aim at a combination of both ideas, which means to start with a relaxation of the problem and then use shape optimization based on level-sets as a rounding strategy.
First we apply a semismooth Newton optimization algorithm to come up with a continuous parameter distribution in the range between zero and one.
Then we use this function as a level-set, such that the mean value of the parameter describes the shape of the interface.
This is a simple rounding strategy which usually leads to poor results.
We then proceed with a shape optimization based on this level-set function.
The benefit here is, that only small deformations have to be made since the rounded minimizer of the relaxation is a good initial guess in terms of location and number of possible shapes.

From a computational point of view it is attractive that the initial guess for the interface of interest does not have to be of very high resolution.
We thus present numerical results, where the relaxation problem is solved on a coarse grid only.
This solution is then interpolated to finer grids and a high resolution shape optimization algorithm is applied.

This work has the following structure:
In Section \ref{sec:model} we introduce the model equations and the associated optimization problems.
Section \ref{sec:algorithm} presents two approaches, a relaxation solver together with a rounding strategy and a shape optimization method, which are combined into one algorithm.
Finally, in Section \ref{sec:numerics} we demonstrate our algorithm in two and three dimensional, numerical test cases.

\section{Model equations and problem formulation}
\label{sec:model}
In this work we focus on a fluid flow in a porous medium driven by advection and diffusion, which transports a tracer.
In a bounded Lipschitz domain $\Omega \subset \mathbbm{R}^d$, with dimension $d \in \lbrace2,3\rbrace$, $u: \Omega \rightarrow \mathbbm{R}$ measures the concentration of the tracer.
It is transported along the velocity field $\velocity \in W^{1, \infty}\left(\Omega, \mathbbm{R}^d\right)$ with diffusivity constant $c \in L^{\infty}(\Omega)$.
With this model we simulate a groundwater flow in a domain with radioactive sources given by $f \in L^2(\Omega)$.
Therefore, $\Omega$ is subdivided into two areas, $\Oout$ and $\Oint$, a clean and a contaminated one.
The source $f$ depends on the shape of $\Oint$.
We further assume that concentration measurements of the tracer $u$ are given in $\Om \subseteq \Omega$.
Figure \ref{fig:domain} gives a sketch of this situation.
\begin{figure}
\begin{center}
\def\svgwidth{0.7\textwidth}
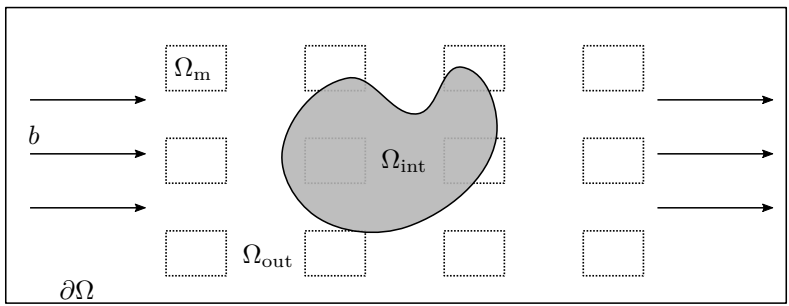
\end{center}
\caption{The flow domain $\Omega = \overline{\Oint} \cup \Oout$ with outer boundary $\partial \Omega$, velocity field $\velocity$ and the location of data measurements $\Om$}
\label{fig:domain}
\end{figure}
The corresponding model is described by the stationary advection-diffusion equations
\begin{equation}
\begin{aligned}
-c \Delta u + \mathrm{div} (\velocity u) &= f \quad\;\; \mathrm{in} \; \Omega\\[0.2cm]
u \velocity^T n - c \frac{\partial u}{\partial n} &= g u \quad \mathrm{on} \; \partial \Omega.
\end{aligned}
\label{eq:strong_form_advection_diffusion}
\end{equation}
The tracer is assumed to enter and leave the domain under the influence of both advection and diffusion, which is expressed as boundary conditions $g = \max\left\lbrace 0,\, \velocity^T n\right\rbrace$, where $n$ denotes the outward-pointing normal vector on $\partial \Omega$.
We obtain the weak formulation by multiplying with test functions $w$ and integrating over $\Omega$ which yields
\begin{equation*}
\int_{\Omega} -c \Delta u w +  w\, \mathrm{div}(\velocity u) \; dx = \int_\Omega fw\; dx.
\end{equation*}
By applying Green's identity we obtain
\begin{equation*}
\int_{\Omega} c (\nabla u)^T \nabla w - u \velocity^T \nabla w \; dx + \int_{\partial \Omega} w(u \velocity^T n  - c \frac{\partial u}{\partial n})\;ds  =  \int_\Omega fw\; dx
\end{equation*}
and finally integration by parts yields the problem in weak form:\\
Find $u\in H^1(\Omega)$ such that
\begin{equation}
\underbrace{\int_{\Omega} c (\nabla u)^T \nabla w -u \velocity^T \nabla w \; dx + \int_{\partial \Omega} wgu\;ds}_{\eqqcolon a_1(u,w)}  = \underbrace{\int_\Omega fw\; dx}_{\eqqcolon l_1(w)}
\label{eq:state}
\end{equation}
for all $w\in H^1(\Omega)$.

The original optimization problem for given data measurements $\bar{u}$ is of the form
\begin{equation}
\begin{aligned}
&\min\limits_{f}\; \frac{ 1}{ 2}\int_\Om (u-\bar{u})^2\; dx\\[0.3cm]
\text{s.t. }& f(x) \in \lbrace 0, 1\rbrace \; \text{ a.e. in $\Omega$ and \eqref{eq:state}}
\end{aligned}
\label{eq:original_problem}
\end{equation}
The binary variable $f$ reflects the situation, that the material in the flow domain can either be contaminated or not.
Values in $(0,1)$ are not allowed.
From a computational point of view this problem is challenging, especially when high spatial resolutions are preferred.
Moreover, a discrete optimization approach produces solutions which typically depend on the chosen discretization.

In the following we reformulate problem \eqref{eq:original_problem} into two stages.
First a relaxed version where intermediate values are allowed as
\begin{equation}
\begin{aligned}
&\min\limits_{f \in L^2(\Omega)}\; \frac{ 1}{ 2}\int_\Om (u-\bar{u})^2\; dx  +  \frac{\mu}{2} \int_\Omega f^2 \; dx\\[0.3cm]
\text{s.t. }&0 \leq f \leq 1 \;\; \text{a.e. in $\Omega$ and \eqref{eq:state}}.
\end{aligned}
\label{eq:relaxed_problem}
\end{equation}
Here the second term is a regularization for $\mu > 0$.
A solution of problem \eqref{eq:relaxed_problem} leaves open the question for a rounding strategy of $f$ towards 0 or 1 in order to approximated the original problem \eqref{eq:original_problem}.

\begin{sloppypar}
Given a solution $f$ of \eqref{eq:relaxed_problem} we define the shape of interest by the rounding strategy ${\Oint = \lbrace x \in \Omega \colon f(x) > 0.5 \rbrace}$.
In order to improve the result, thereafter we turn to the following optimization problem:
\end{sloppypar}
\begin{equation}
\begin{aligned}
&\min\limits_{\Omega}\; J(u,\Omega) \coloneqq \frac{ 1}{ 2}\int_\Om (u-\bar{u})^2\; dx\\[0.3cm]
\text{s.t. }&\text{\eqref{eq:state} with } f(x) = \begin{cases} 1,\; x\in \Oint\\ 0,\; x \in \Omega\setminus\Oint \end{cases}
\end{aligned}
\label{eq:shape_problem}
\end{equation}

Note that problem \eqref{eq:shape_problem} is restricted to the original state equation \eqref{eq:state} with a modified right hand side, which now depends on the shape of $\Oint$.
In this problem formulation it is minimized over $\Omega$ and not $\Oint$, which determines the shape of the source $f$.
Later, descent directions will be continuous deformations defined on $\Omega$, which are zero at $\partial \Omega$ such that the outer shape is fixed.
Thus, $\Oint$ is simultaneously changed together with $\Omega$.

If we assume a reasonable solution of the relaxed problem \eqref{eq:relaxed_problem} and an efficient rounding strategy, we have a good initial guess for \eqref{eq:shape_problem} and only small corrections to the shape of the distribution of $f$ have to be made.

In the following we formulate the Lagrangian of problem \eqref{eq:shape_problem} where we indicate the dependence of $f$ on $\Omega$ by $f(\Omega)$ and obtain
\begin{multline}
G \left(u,w,f(\Omega)\right) = \frac{ 1}{ 2}\int_\Om (u-\bar{u})^2\; dx + \int_{\Omega} c (\nabla u)^T \nabla w -u \velocity^T \nabla w \; dx\\
+ \int_{\partial \Omega} wgu\;ds - \int_\Omega fw\; dx.
\label{eq:lagrangian}
\end{multline}
Differentiating \eqref{eq:lagrangian} with respect to the state variable $u$ in the sense
\begin{equation*}
\frac{d}{d t} G \left(u + tz,w,f(\Omega)\right)\Bigr|_{t=0} = 0 \quad \forall\; z \in H^1(\Omega)
\end{equation*}
then yields the weak formulation of the adjoint equation:
Find $w\in H^1(\Omega)$ such that
\begin{equation}
\int_{\Omega} c (\nabla w)^T \nabla z - z \velocity^T \nabla w \; dx + \int_{\partial \Omega} zgw\;ds = -\int_\Om (u-\bar{u})z\; dx
\label{eq:adjoint}
\end{equation}
for all $z\in H^1(\Omega)$.

Up to now we have considered the general case of arbitrary velocity fields, which is important for the presentations in Section \ref{sec:algorithm}.
For the physical model we make the assumption of incompressiblity of the velocity $\velocity$.
\begin{myremark}
\label{rmrk:adjoint}
If we choose a velocity field that fulfills $\mathrm{div}(\velocity) = 0$, the state equation simplifies to
\begin{equation*}
-c \Delta u + \velocity^T \nabla u = f
\end{equation*}
due to the identity $\mathrm{div}(\velocity u) = u\, \mathrm{div}(\velocity) + \velocity^T \nabla u$.
In that case we can apply Green's identity to the adjoint equation \eqref{eq:adjoint} and obtain
\begin{multline*}
\int_{\Omega} - c z \Delta w - z \velocity^T \nabla w \; dx + \int_{\partial \Omega} z(gw + c \frac{\partial w}{\partial n} + w \velocity^T n - w \velocity^T n)\;ds
= -\int_\Om (u-\bar{u})z\; dx.
\end{multline*}
From that we can derive the adjoint PDE in strong from
\begin{align*}
-c \Delta w + (- \velocity^T) \nabla w &= - \chi_\Om (u -\bar{u}) \quad\;\;\; \mathrm{in} \; \Omega\\[0.2cm]
w (-\velocity^T) n - c \frac{\partial w}{\partial n} &= \underbrace{(g - \velocity^T n)}_{= \max\left\lbrace 0, -\velocity^Tn\right\rbrace} w \quad \mathrm{on} \; \partial \Omega.
\end{align*}
Note that in this case the same assembly routine can be used with a different direction of the velocity field.
\end{myremark}
In the case $\mathrm{div}(\velocity) = 0$ and $g = \max\left\lbrace 0,\, b^T n\right\rbrace$ one can proof existence and uniqueness of weak solutions.
This yields the invertability of the control-to-state map, which is an important ingredient for the optimization algorithms in the next section.
\begin{myremark}
\label{rmrk:coercivity}
In order to apply Lax Milgram theorem, which yields existence and uniqueness of solutions of \eqref{eq:state}, we first observe the equivalence
\begin{equation*}
\begin{aligned}
&- \int_{\Omega} u b^T \nabla u\; dx = \int_\Omega \underbrace{u\, \mathrm{div}(bu)}_{ = u^2 \mathrm{div}(b) + ub^T \nabla u}\; dx\; - \int_{\partial \Omega} u^2 b^T n \; ds \\
\Leftrightarrow\quad & - \int_{\Omega} u b^T \nabla u\; dx = - \frac{1}{2} \int_{\partial \Omega} u^2 b^T n \; ds.
\end{aligned}
\end{equation*}
By assuming $\Gamma_1 \coloneqq \lbrace x\in \partial \Omega : \left\vert b^T n \right\vert \geq \eta \rbrace$ for $\eta > 0$ and $\left\vert \Gamma_1 \right\vert > 0$ the generalized Friedrichs inequality (cf.~\cite[Chapter 2]{troeltzsch2010optimal}) yields a constant $r>0$, which is independent of $u$, such that
\begin{equation*}
\begin{aligned}
a_1(u,u) &=  \int_{\Omega} c (\nabla u)^T \nabla u \; dx - \frac{1}{2} \int_{\partial \Omega} u^2 b^T n \; ds \; + \; \int_{\partial \Omega} g u^2 \; ds\\
&= \int_{\Omega} c (\nabla u)^T \nabla u \; dx + \frac{1}{2} \int_{\partial \Omega} \left\vert \velocity^T n \right\vert u^2 \; ds\\
&\geq \min\left\lbrace c,\, \frac{\eta}{2}\right\rbrace \left(\int_{\Omega} (\nabla u)^T \nabla u \; dx + \int_{\Gamma_1} u^2 \; ds\right) \geq r\Vert u \Vert_{H^1(\Omega)}\\
\end{aligned}
\end{equation*}
which yields the coercivity of $a_1$.
\end{myremark}
Note that under the assumptions in Remark \ref{rmrk:adjoint} and \ref{rmrk:coercivity} also the adjoint equation admits a unique solution.

\section{Optimization algorithm}
\label{sec:algorithm}

The idea we follow in this section is to combine solutions of the problems \eqref{eq:relaxed_problem} and \eqref{eq:shape_problem} in order to obtain an approximate solution for \eqref{eq:original_problem}.
Note that a discussion on a regularization strategy for the binary problem \eqref{eq:original_problem} is thereby shifted to the relaxation and shape optimization problem.
For simplicity we discuss the relaxation problem in a finite dimensional setting.
This means that the optimization operates on discretized PDEs.
After that we introduce the shape optimization problem and show the link between both.

\subsection{A continuous optimization approach}
\label{sec:continuous_optimization}

In our algorithm we compute an initial guess of the parameter distribution by solving the relaxed problem \eqref{eq:relaxed_problem}.
This is in the class of PDE-constraint optimal control problems with bounds on the control.
In this subsection we concentrate on a finite dimensional setting.
This means that we apply the semi-smooth Newton optimization algorithm on the discretization of \eqref{eq:relaxed_problem}.
For a discussion on optimization in function spaces see for instance \cite{hintermuller2002primal,weiser2008control,ulbrich2002semismooth}.
A survey on methods for similar problems can be found in \cite{herzog2010algorithms}.

In the following we assume $V_h \subset H^1(\Omega)$ to be the standard finite element space with Lagrange polynomials on a triangulation or structured mesh in $\Omega$.
Let $M$ be the mass matrix and $S$ the stiffness matrix corresponding to the model problem \eqref{eq:state}, both in terms of $V_h$.
Further, let $\tilde{M}$ be a reduced mass matrix in the sense $\int_\Om p_1 p_2\;dx$ for all $p_1, p_2 \in V_h$.
We can then formulate the finite dimensional optimization problem
\begin{equation*}
\begin{aligned}
&\min\limits_{f \in V_h}\; J(u,f) \coloneqq \frac{ 1}{ 2} \left(u-\bar{u}\right)^T \tilde{M} \left(u-\bar{u}\right) + \frac{\mu}{2} f^T M f\\[0.3cm]
\text{s.t. }& Su = Mf \;\; \text{and}\;\; 0 \leq f \leq 1. \\
\end{aligned}
\end{equation*}
From this we can derive the reduced cost functional
\begin{equation*}
\hat{J}(f) \coloneqq \frac{1}{2} \left[ \left( S^{-1}Mf\right)^T \tilde{M} S^{-1} M f - 2 \left( S^{-1}Mf\right)^T \tilde{M} \bar{u} + \bar{u}\tilde{M} \bar{u} + \mu f^T M f \right]
\end{equation*}
and further by differentiation with respect to $f$ the reduced gradient
\begin{equation*}
\nabla \hat{J}(f) = M^T S^{-T} \tilde{M} S^{-1} M f - M^T S^{-T} \tilde{M} \bar{u} + \mu M f
\end{equation*}
and the Hessian
\begin{equation*}
\nabla^2 \hat{J}(f) = M^T S^{-T} \tilde{M} S^{-1} M + \mu M.
\end{equation*}
Note that $S^{-1}$ and $S^{-T}$ are only used as symbols for the state \eqref{eq:state} and adjoint \eqref{eq:adjoint} solver.
We apply for both PDEs an iterative solver, which is further discussed in Section \ref{sec:numerics}.
Similarly a multiplication with the mass matrix $M$ means the assembly of a volumic source term.
This matrix is not built explicitly.

Let $\mathcal{J}$ be the index set of the degrees of freedom in $V_h$.
Then we define the active set for upper and lower bounds
\begin{equation}
\begin{aligned}
\mathcal{A}^k_+ \coloneqq \lbrace j \in \mathcal{J} \colon \lambda^k + \gamma \left( f^k - f_+ \right) > 0 \rbrace\\
\mathcal{A}^k_- \coloneqq \lbrace j \in \mathcal{J} \colon \lambda^k + \gamma \left( f^k - f_- \right) < 0 \rbrace
\end{aligned}
\label{eq:active_inactive_set}
\end{equation}
and the active set $\mathcal{A}^k \coloneqq \mathcal{A}^k_+ \cup \mathcal{A}^k_-$, respectively.
Further, let $\mathcal{I}^k \coloneqq \mathcal{J}\setminus\mathcal{A}^k$ denote the inactive set.
With $\chi_{\mathcal{A}^k}, \chi_{\mathcal{A}^k_-}$ and $\chi_{\mathcal{A}^k_+}$ we denote restriction operators to the active sets.

We can then formulate the semismooth Newton system, which reflects the primal-dual active set algorithm
\begin{equation*}
\begin{bmatrix}
\nabla^2 \hat{J} (f^k) & I\\
\gamma \chi_{\mathcal{A}^k}  & - \chi_{\mathcal{I}^k}
\end{bmatrix}
\begin{pmatrix}
\delta f\\
\delta \lambda
\end{pmatrix}
= -
\begin{pmatrix}
\nabla \hat{J}(f^k) + \lambda^k\\
\gamma \chi_{\mathcal{A}^k_{+}} (f^k - f^+) + \gamma \chi_{\mathcal{A}^k_{-}} (f^k - f^-) - \chi_{\mathcal{I}^k} \lambda^k
\end{pmatrix}
\end{equation*}
for the case $\gamma = \mu$.
For a further discussion on the choice of the parameter $\gamma$ see \cite{bergounioux1997augemented}.
This system can be transformed to a symmetric version
\begin{equation}
\begin{bmatrix}
\nabla^2 \hat{J} (f^k) & \chi_{\mathcal{A}^k}^T\\
\chi_{\mathcal{A}^k}  & 0
\end{bmatrix}
\begin{pmatrix}
\delta f\\
\delta \lambda_{\mathcal{A}^k}
\end{pmatrix}
= -
\begin{pmatrix}
\nabla \hat{J}(f^k) + \chi_{\mathcal{A}^k} \lambda^k\\
\chi_{\mathcal{A}^k_{+}} (f^k - f^+) + \chi_{\mathcal{A}^k_{-}} (f^k - f^-)
\end{pmatrix}
\label{eq:symmetric_semismooth_newton}
\end{equation}
where $\delta \lambda_{\mathcal{A}^k} \in \mathbbm{R}^{\left\vert\mathcal{A}^k\right\vert}$ is the restriction of Lagrange multiplier $\delta \lambda$ to the active set.
Updates are then computed via $\lambda^{k+1} \gets \lambda^k + \chi_{\mathcal{A}^k}^T \delta \lambda_{\mathcal{A}^k}$ and $\lambda_{\mathcal{I}^k}^{k+1} \gets 0$.
This is computationally more favorable since it allows to use MINRES instead of {GMRES} as linear solver.
Yet, it should be remarked that the dimension of system \eqref{eq:symmetric_semismooth_newton} changes in every iteration $k$ when the number of indices in the active set $\mathcal{A}^k$ changes.

The advantage of an iterative Krylov method is, that we do not need to construct the system matrix in \eqref{eq:symmetric_semismooth_newton}, which would require to compute $S^{-1}$ and $S^{-T}$.
Only a function for matrix vector products is necessary, which then invokes the iterative solvers for state \eqref{eq:state} and adjoint \eqref{eq:adjoint} equations.

\subsection{Shape optimization approach}
\label{sec:shape_optimization}
In the previous subsection we describe an algorithm that yields a solution to the relaxation of our original optimization problem.
This minimizer takes values in the interval $[0, 1]$.
For the moment we assume that we have a rounding strategy, which determines regions where the source term (right hand side in problem \eqref{eq:original_problem}) is active and where not.
This strategy is later described in more detail.
We now concentrate on a gradient descent for the shape problem \eqref{eq:shape_problem}, which is subsequent to \eqref{eq:relaxed_problem}.

We start with the definition of shape functionals like in the objective function in problem \eqref{eq:shape_problem}.
\begin{mydefinition}
\label{def:shape_functional}
Let $ D \subset \mathbbm{R}^d$, $D \neq \emptyset$ for $d \in \mathbbm{N}$, $d \geq 2$ and $A \subset \lbrace \Omega \colon\, \Omega \subset D\rbrace $.
A function
\begin{equation*}
J \colon A \rightarrow \mathbbm{R},\, \Omega \mapsto J(\Omega)
\end{equation*}
is called a shape functional.
\end{mydefinition}
In order to analyze sensitivities of shape functionals with respect to deformations of the underlying shape we define the perturbation of a domain.
\begin{mydefinition}
\label{def:perturbed_domain}
Let $D$ be as in Definition \eqref{def:shape_functional} and $\lbrace F_t \rbrace_{t \in [0,T]}$ for $T > 0$ a family of mappings $F_t \colon\, \overline{D} \rightarrow \mathbbm{R}^d$ such that $F_0$ is the identity on $\mathbbm{R}^d$ restricted to $\overline{D}$.
We define the perturbation of a domain $\Omega \subset D$ by
\begin{equation*}
\Omega_t \coloneqq \lbrace F_t(x) \colon\, x\in \Omega\rbrace
\end{equation*}
and the perturbed boundary $\Gamma_t$ of $\Gamma = \partial \Omega$ analogously.
\end{mydefinition}
In the field of shape optimization there are two common choices for $F_t$ (for further reading see \cite{sokolowski1992introduction,delfour2001shapes}).
Given a differentiable vector field $v : \mathbbm{R}^d \rightarrow \mathbbm{R}^d$ the perturbation of identity has the following form
\begin{equation}
F_t(x) = x + tv(x).
\label{eq:perturbation_of_identity}
\end{equation}
Alternatively, the velocity method, where shapes are assumed to be particles following a transport equation in terms of the velocity field $v$ as
\begin{equation*}
\begin{aligned}
\frac{\partial \xi_x(t)}{\partial t} &= v\left(\xi_x(t)\right)\, , \; t \in [0, T]\\
\xi_x(0) &= x\\
F_t(x) &= \xi_x(t).
\end{aligned}
\label{eq:velocity_method}
\end{equation*}
The proof of Theorem \ref{thm:shape_derivative} is based on the perturbation of identity.
Yet, in Section \ref{sec:combined_optimization} and \ref{sec:numerics} we will focus on deformations according to the velocity method, which is a natural choice when dealing with level-sets.

The sensitivity analysis in this work makes use of the concept of material derivatives.
\begin{mydefinition}
\label{def:material_derivative}
Let $\Omega, \Omega_t, F_t$ be as in Definition \ref{def:perturbed_domain} and $ \lbrace u_t \colon \Omega_t \rightarrow \mathbbm{R}, t<T\rbrace$ a family of functions on perturbed domains. The material derivative of $p(=p_0)$ in $x\in \Omega$ is given by
\begin{equation*}
\dt{u}(x) = \lim\limits_{t \searrow 0} \frac{(u_t \circ F_t )(x) - u (x)}{t}.
\end{equation*}
In the following we use both $\dt{u}$ and $D_m(u)$ for the material derivative of $u$.
\end{mydefinition}
Finally, the directional derivative of a shape is defined as follows:
\begin{mydefinition}
\label{def:shape_derivative}
Let again $ D \subset \mathbbm{R}^d$, $D \neq \emptyset$ for $d \in \mathbbm{N}$, $d \geq 2$ be open.
Further, let $\Omega \subset D$ and $J$ be a shape functional according to Definition \ref{def:shape_functional}.
The shape derivative of $J$ at $\Omega$ in direction $v \in C_0^1(D, \mathbbm{R}^d)$ is defined as
\begin{equation*}
dJ(\Omega)[v] \coloneqq \lim_{t \searrow 0} \frac{J(\Omega_t) - J(\Omega)}{t}.
\end{equation*}
\end{mydefinition}
With these definitions we can formulate the directional derivative of the objective function in the shape optimization problem.
\begin{mytheorem}
\label{thm:shape_derivative}
Let $u \in H^1(\Omega)$ denote the weak solution of \eqref{eq:state}, which additionally fulfills the constraints in \eqref{eq:shape_problem}, with a source $f$ depending on $\Oint$.
Further, let $w \in H^1(\Omega)$ be the weak solution of the adjoint equation \eqref{eq:adjoint}.
The shape deriviative of the functional $J$ in direction of a vector field $v \in C_0^1(D, \mathbbm{R}^d)$ is given by
\begin{multline}
dJ(u, \Omega)[v] = \int_\Omega - c(\nabla u)^T \left(\nabla v + (\nabla v)^T \right)\nabla w - u \velocity^T(\nabla v)^T \nabla w \\
+ \mathrm{div} (v) \left[\frac{1}{2} \chi_{\Om}(u-\bar{u})^2 + c(\nabla  u)^T \nabla w - u \velocity^T \nabla w -fw \right] dx.
\label{eq:shape_derivative}
\end{multline}
\end{mytheorem}
\begin{proof}

Following the presentation in \cite[Chapter 10, Section 5.2]{delfour2001shapes} we reformulate the objective function $J$ as the solution of the saddle point problem
\begin{equation}
J(u, \Omega) = \adjustlimits\inf_{u \in H^1(\Omega)} \sup_{w \in H^1(\Omega)} G(u,w,\Omega)
\label{eq:inf_sup}
\end{equation}
where $w$ plays the role of the Lagrange multiplier.
We then apply Theorem 2.1 in \cite{correa1985directional} in order to differentiate the right hand side of \eqref{eq:inf_sup} in direction of a vector field $v \in C_0^1(D, \mathbbm{R}^d)$ according to the transformation \eqref{eq:perturbation_of_identity}.
The following two identities show the material derivative of integrated quantities for both a volume integral over $\Omega$ and a surface integral over $\partial \Omega$
\begin{equation*}
\begin{aligned}
D_m \left( \int_{\Omega_t} p_t\, dx \right) &= \int_\Omega \dt{p} + \mathrm{div} (v) p\, dx\\
D_m \left( \int_{\partial \Omega_t} p_t\, dx \right) &= \int_{\partial \Omega} \dt{p} + \mathrm{div}_{\partial \Omega} (v) p\, dx.
\end{aligned}
\end{equation*}
For more details and a proof we refer the reader to \cite{haslinger2003advancecs,welker2016efficient}.
We then obtain the shape derivative at $\Omega$ in direction $v$
\begin{multline}
dG\left(u,w,f(\Omega)\right)[v] = \frac{d}{dt} G \left(u_t,w_t, f(\Omega_t)\right)\Bigr|_{t=0} =\\
\int_\Omega D_m \left[ \frac{1}{2} \chi_{\Om}(u-\bar{u})^2 + c(\nabla  u)^T \nabla w - u\velocity^T \nabla w -fw \right]\\
+ \mathrm{div} (v) \left[\frac{1}{2} \chi_{\Om}(u-\bar{u})^2 + c(\nabla  u)^T \nabla w - u\velocity^T \nabla w -fw \right]\; dx\\
+ \int_{\partial \Omega} D_m (wgu) + \mathrm{div}_{\partial \Omega}(v) w g u \; ds.
\label{eq:pre_shape_derivative}
\end{multline}
Under the assumption that $v$ vanishes in a small neighborhood of $\partial \Omega$ we can neglect $\mathrm{div}_{\partial \Omega}(v)$.
This is reasonable due to the fact, that the outer shape of $\Omega$ is not intended to be a variable in the optimization.
Utilizing rules for the material derivatives (cf.\ \cite{welker2016efficient}) it follows that
\begin{align*}
D_m(\frac{1}{2}\chi_{\Om}\left(u-\bar{u})^2\right) &= \chi_{\Om}\dt{u}(u-\bar{u})\\
D_m\left(c(\nabla  u)^T \nabla w\right) &= c(\nabla \dt{u})^T \nabla w + c(\nabla u)^T \nabla \dt{w} - c(\nabla u)^T (\nabla v + (\nabla v)^T)\nabla w\\
D_m\left(u\velocity^T \nabla w\right) &= u(\velocity^T\nabla \dt{w} - \velocity^T(\nabla v)^T \nabla w) + \dt{u}\velocity^T \nabla w\\
D_m\left(fw\right) &= \dt{f}w + f\dt{w}.
\end{align*}
Plugging this into equation \eqref{eq:pre_shape_derivative} and verifying that $\dt{f} = \dt{\chi_\Oint} = 0$ we obtain
\begin{multline*}
dG\left(u,w,f(\Omega)\right)[v] =\\
\int_\Omega \chi_{\Om}\dt{u}(u-\bar{u}) + c(\nabla \dt{u})^T \nabla w + c(\nabla u)^T \nabla \dt{w} - c(\nabla u)^T \left(\nabla v + (\nabla v)^T \right)\nabla w\\
- u(\velocity^T\nabla \dt{w} - \velocity^T(\nabla v)^T \nabla w) - \dt{u}\velocity^T \nabla w - \dt{f}w - f\dt{w} \; dx\\
+ \mathrm{div} (v) \left[\frac{1}{2} \chi_{\Om}(u-\bar{u})^2 + c(\nabla  u)^T \nabla w - \velocity^T \nabla w) -fw \right] dx+ \int_{\partial \Omega} D_m (wgu)\; ds.
\end{multline*}
Since $u$ and $w$ are solutions to \eqref{eq:state} and \eqref{eq:adjoint}, respectively, we can plug $\dt{u}$ and $\dt{w}$ in the role of test functions and obtain \eqref{eq:shape_derivative}.
\end{proof}

Note that in the formulation of the shape optimization problem \eqref{eq:shape_problem} a set of feasible shapes is not specified.
This discussion is beyond the scope of this work and can be found for instance in \cite{schulz2016efficient}.
Here we assume that the inner product of a suitable shape space is given by the following bilinear form:
Let ${a_2 : H^1\left(\Omega, \mathbbm{R}^d\right) \times H^1\left(\Omega, \mathbbm{R}^d\right) \rightarrow \mathbbm{R}}$ be given by
\begin{equation}
a_2(p,q) \coloneqq (1-\alpha) \int_\Omega \sum_{i=1}^d (\nabla p_i)^T \nabla q_i \, dx + \alpha \int_\Omega \sum_{i=1}^d p_i\, q_i\, dx
\label{eq:H1_bilinear_form}
\end{equation}
for an $\alpha \in (0,1)$.
We then solve
\begin{equation}
a(\psi, v) = dJ(u,w,\Omega)[v] \quad \forall\; v \in H^1\left(\Omega, \mathbbm{R}^d\right)
\label{eq:gradient_representation}
\end{equation}
in order to obtain a representation $\psi$ for the shape gradient.
In \cite{schulz2016computational} the bilinear form of linear elasticity is chosen and advantages with respect to mesh quality are discussed.
Yet, here we concentrate on \eqref{eq:H1_bilinear_form} since no boundary conditions have to be specified.
This is important since we assume that the position of shapes are unknown.
It is thus possible that shapes are close to the outer boundary of $\Omega$.
If we specify a bilinear form for the gradient representation, which is based on, e.g., homogeneous Dirichlet conditions, the deformation field $\psi$ is influenced in the neighborhood of the shapes in an unintended manner.

\subsection{A combined optimization algorithm}
\label{sec:combined_optimization}
A common approach (see for instance \cite[Section 6.3]{mohammadi2001applied}) would be to use the vector field $\psi$ as a deformation to the finite element mesh in $\Omega$.
This deformation can either be in the form of the perturbation of identity \eqref{eq:perturbation_of_identity} or the velocity method \eqref{eq:velocity_method}.
Assuming that the interface between $\Oout$ and $\Oint$ is resolved within the finite elements, this approach can be used as a descent direction for the objective in \eqref{eq:shape_problem}.
The advantage of this approach is that $\Oint$ and the source term $f$ are perfectly resolved within the mesh in each optimization iteration.
On the other hand, one has to steadily reassemble all arising discretization operators for the PDEs.

Alternatively, we follow the so called level-set approach (see for instance \cite{allaire2004structural}).
The interface between $\Oint$ and $\Oout$ is implicitly prescribed by a level-set of a function $\phi : \Omega \rightarrow \mathbbm{R}$ in the form
\begin{equation*}
x \in \begin{cases}
\Oout &\;\mathrm{if}\; \phi(x) < 0\\
\partial \Oint &\;\mathrm{if}\; \phi(x) = 0\\
\Oint &\;\mathrm{if}\; \phi(x) > 0.
\end{cases}
\end{equation*}
The advantage of this approach is, that interfaces do not have to be resolved in the mesh.
As a result structured grids are possible for the discretization of the PDEs leading to major computational advantages.
Yet, the interpretation of $\psi$ as a descent direction is more complicated compared to a mesh deformation.
For this purpose we define the level-set function ${\phi: \Omega \times \left( 0 , T \right] \rightarrow \mathbbm{R}}$ to be a density distribution in an advection diffusion transport equation
\begin{equation}
\begin{aligned}
\frac{\partial \phi}{\partial t} - \epsilon \Delta \phi + \mathrm{div} (\psi \phi) &= 0 \quad\;\; \mathrm{in} \; \Omega \times \left( 0 , T \right]\\[0.2cm]
\phi \psi^T n - \epsilon \frac{\partial \phi}{\partial n} &= g \phi \quad \mathrm{on} \; \partial \Omega \times \left( 0 , T \right]\\[0.2cm]
\phi &= \tilde{\phi} \quad\;\; \mathrm{in} \; \Omega \times \lbrace 0 \rbrace
\end{aligned}
\label{eq:time_dependent_adv_diff}
\end{equation}
where again $g = \max\lbrace 0, \psi^T n\rbrace$ is chosen.
The associated variational form is given by:

Find $\phi \in L^2\left(0,\, T\colon\, H^1(\Omega)\right)$ with $\phi_t \in L^2\left(0,\, T\colon\, H^1(\Omega)^\prime \right)$ such that
\begin{equation}
\int_{\Omega} p \phi_t(t)\; dx \;+\; \int_{\Omega} c (\nabla \phi(t))^T \nabla p - \phi(t) \psi^T \nabla p \; dx\; +\; \int_{\partial \Omega} p g \phi(t)\;ds = 0
\label{eq:weak_time_advection_diffusion}
\end{equation}
for all $p \in H^1(\Omega)$, for $t$ pointwise a.e.\ in $(0,\,T)$ and $\phi(0) = \tilde{\phi}$.

Note that the diffusive part in this model might look surprising, since we want a pure transport of the level-set function $\tilde{\phi}$.
Yet, the small influence of diffusion, which is controlled by a factor $\epsilon$, leads to smooth shapes and makes the separation and union of shapes possible.
This is further discussed in Section \ref{sec:numerics}.
The shape gradient $\psi$ enters the system as a time-dependent velocity field.
The velocity $\psi$ is assumed to be piece-wise constant in time and changes when a new descent direction in the optimization problem is computed.
\begin{myremark}
Note that the weak formulation requires $\psi \in L^\infty (\Omega)$ which we can not guarantee.
Since $\psi$ is the weak solution of \eqref{eq:gradient_representation}, further assumptions on the shape derivative $dJ(u,w,\Omega)[v]$ would be required.
Yet, in the discretized problem for a finite dimensional subspace the combination of \eqref{eq:gradient_representation} and the transport equation of the level set \eqref{eq:weak_time_advection_diffusion} provides the desired result.
\end{myremark}
From a computational point of view it is a crucial advantage that the shape derivative \eqref{eq:shape_derivative} and gradient $\psi$, given by \eqref{eq:gradient_representation}, can be completely described by volume terms.
It is thus not necessary to reconstruct the implicit surface ${\lbrace x \in \Omega : \phi(x) = 0\rbrace}$ at any point in the algorithm.
However, care has to be taken by the integration of the right hand side in \eqref{eq:state} since $f = \chi_{\Oint}$ is not representable in $V_h$.
The choice of appropriate quadrature rules is discussed in Section \ref{sec:numerics}.
\begin{algorithm}
\caption{Combined optimization algorithm}
\label{alg1}
\begin{algorithmic}[1]
\itemsep0.05cm
\STATE Choose $f^0=0.5$ as initial guess
\STATE $k \gets 0$
\WHILE{$\left\Vert \begin{pmatrix}\delta f\\ \delta \lambda_{\mathcal{A}^k}\end{pmatrix} \right\Vert_2 > \epsilon_1$}
	\itemsep0.05cm
	\STATE Evaluate $\mathcal{A}^k_+$, $\mathcal{A}^k_-$ and $\mathcal{I}^k$ according to \eqref{eq:active_inactive_set}
	\STATE Solve system \eqref{eq:symmetric_semismooth_newton} with MINRES
	\STATE $f^{k+1} \gets f^k + \delta f$, $\;\lambda^{k+1} \gets \lambda^k + \chi_{\mathcal{A}^k}^T \; \delta \lambda_{\mathcal{A}^k}$ and $\lambda_{\mathcal{I}^k}^{k+1} \gets 0$
	\STATE $k \gets k+1$
\ENDWHILE
\STATE Compute initial level-set function $\phi^0(x) \gets \frac{f^k(x) - f^k_\mathrm{min}}{f^k_\mathrm{max} - f^k_\mathrm{min}} - 1$, $x \in \Omega$
\STATE Optional: Interpolate $\phi^0$ from coarse grid $V_H$ to finer grid $V_h$, $H > h$
\STATE $k \gets 0$
\WHILE{$\left\Vert \psi \right\Vert_{L^2(\Omega)} > \epsilon_2$}
	\STATE Solve state equation \eqref{eq:state} for $u$
	\STATE Solve adjoint equation \eqref{eq:adjoint} for $w$
	\STATE Compute shape gradient by solving $a(\psi, v) = dJ(u,w,\Omega)[v]$ for all $v \in \left(V_h\right)^d$ according to \eqref{eq:gradient_representation}
	\STATE Propagate the level-set by solving \eqref{eq:time_dependent_adv_diff} for $(0, T]$ with $\tilde{\phi} = \phi^k$ and set ${\phi^{k+1} \gets \phi(T)}$
	\STATE $k \gets k+1$
\ENDWHILE
\end{algorithmic}
\end{algorithm}

Algorithm \ref{alg1} shows the combination of the two optimization approaches discussed above.
From line 1 to 9 we have the classical primal-dual active set algorithm for the relaxed problem.
Then the initial level-set function $\phi^0$ is defined to reflect the interface $\partial \Oint$ at the mean value of the minimizer $f^k$.
This is done in line 9.
Since $f_\mathrm{min} = 0$ or $f_\mathrm{max} = 1$ is not necessarily fulfilled, we choose this rounding strategy instead of $\partial \Oint = \left\lbrace x \in \Omega : f(x) = 0.5\right\rbrace$.
In any case we obtain $\phi^0 \in [-0.5,\, 0.5]$.
\begin{myremark}
We need additional assumptions for the initial level-set $\phi^0$ to produce a feasible shape $\Oint$.
A minimum $f_\mathrm{min}$ and maximum $f_\mathrm{max}$ has to exist in $\overline{\Omega}$.
Further, the set $\left\lbrace x \in \Omega : \phi^0(x) = 0\right\rbrace$ has to cut out the domain $\Oint$.
At least $\phi^0 \in C(\overline{\Omega})$ would be desirable.
Therefor, we could modify the Lagrangian \eqref{eq:adjoint} by adding the regularization of the relaxed problem and neglect the condition $0 \leq f \leq 1$ leading to $\tilde{G} (u,w,f) = G(u,w,f)  + \frac{\mu}{2} \int_\Omega f^2 \; dx$.
The third part of the optimality system given by the partial derivative of $\tilde{G}$ with respect to $f$ is $\mu f - w = 0$ in strong form.
We can thus transfer the regularity of the adjoint $w$ to $f$, which is under additional assumptions in $H^2(\Omega)$.
Since in this work $\Omega$ is a $d\in \lbrace2,\,3 \rbrace$ dimensional domain, the Sobolev embedding theorem yields $H^2(\Omega) \hookrightarrow C(\Omega)$.
For further discussions we refer the reader to \cite[Section 6.3]{evans1993partial} for the homogeneous Dirichlet case and \cite{nittka2011regularity} for Robin-type boundary conditions.
\end{myremark}

With the initial guess of the location, number and shape of inclusions $\Oint$ we then proceed in line 10 of Algorithm \ref{alg1}.
This step is optional and interpolates $\phi^0$ given in a coarse grid finite element space $V_H$ to a finer space $V_h$.
This is further addressed in Section \ref{sec:numerics}.
The lines 12 to 18 finally implement the shape optimization problem.

\section{Numerical results}
\label{sec:numerics}
In this section we demonstrate the performance of the combined optimization Algorithm \ref{alg1}.
We therefore choose two test environments.
First a two dimensional setting where PDE solutions are computationally cheap and can be done by matrix factorization.
A second test case is a similar setting in three dimensions where we concentrate on the application of iterative solvers.
In both cases we want to emphasis the ability of the algorithm to change the topology of the domain $\Omega$ in the sense, that we are not restricted to the number of inclusions $\Oint$, which are determined as initial guess in line 10 of Algorithm \ref{alg1}.

\begin{figure}
\begin{center}\includegraphics[width=0.9\textwidth]{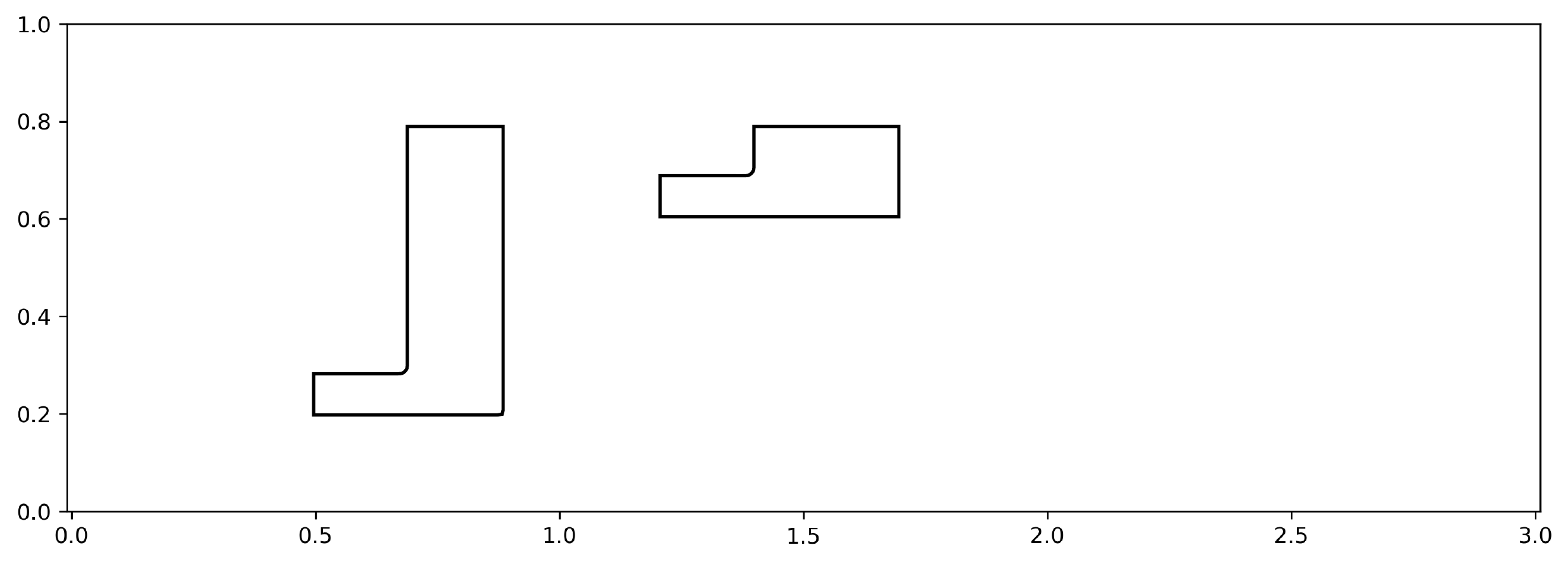}\end{center}
\caption{Shape of the source term for simulated data.}
\label{fig:2dsource}
\end{figure}
In the two dimensional setting the domain $\Omega$ with two inclusion $\Oint$ is visualized in Figure \ref{fig:2dsource}.
The domain is discretized by a $300\times 100$ equidistant, structured mesh and we choose $V_h$ to be the standard bilinear finite element space.
By choosing the source term $f=\chi_\Oint$ we compute a reference solution $u$ according to the model equations \eqref{eq:strong_form_advection_diffusion}.
We then add Gaussian noise to reference state and use it as measurements $\bar{u}$ for the optimization algorithm.
We choose the variance to be 5\% of the maximum value of the PDE state.
Let $u_\mathrm{max}$ be the maximum of the discretized solution $u$ the noise is then chosen according to $\mathcal{N}(0,\, u_\mathrm{max}\cdot 0.05)$.
The reference solution is depicted in Figure \ref{fig:2dstate} and the corresponding measurements in Figure \ref{fig:2dmeasurements}.
In a first test we assume full observability, which means that $\Om = \Omega$.
\begin{figure}
\begin{center}\includegraphics[width=0.9\textwidth]{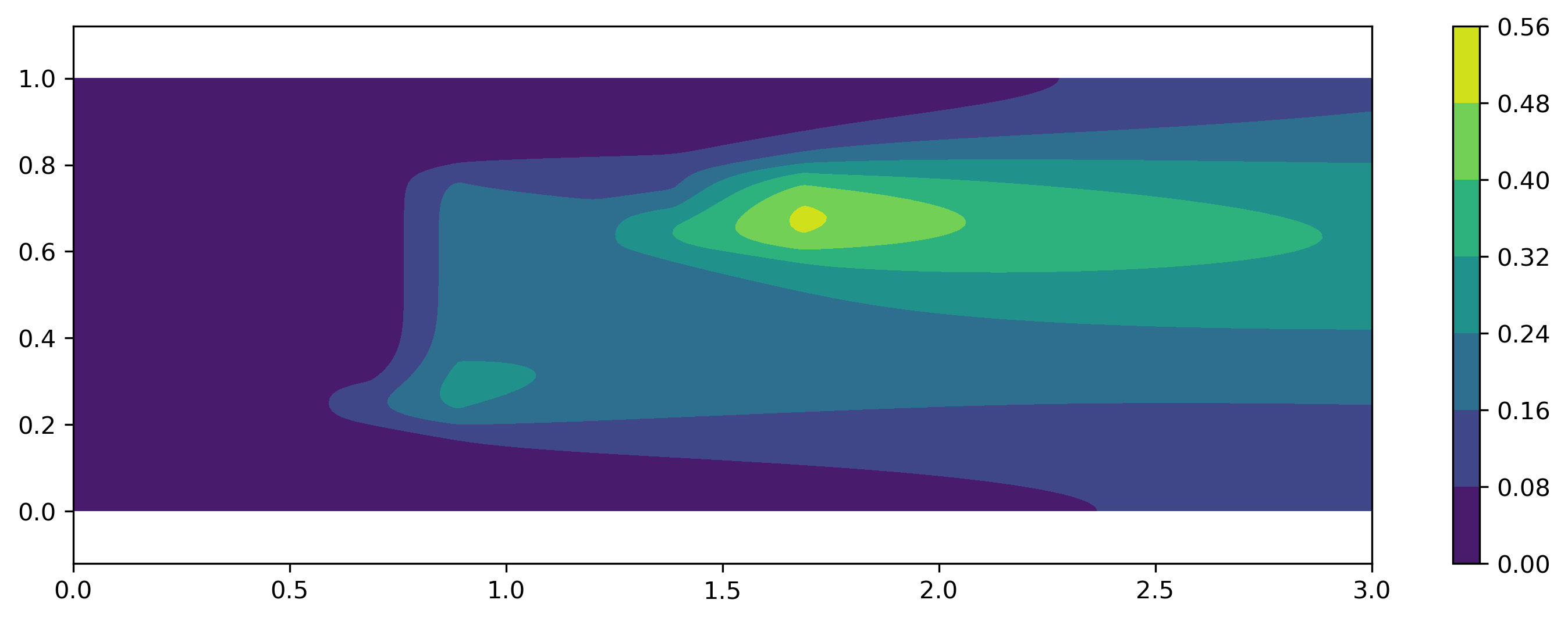}\end{center}
\caption{PDE solution with diffusion coefficient $c=0.01$ and and velocity field $\velocity=(1, 0)^T$ according to the source term depicted in Figure \ref{fig:2dsource}.}
\label{fig:2dstate}
\end{figure}
\begin{figure}
\begin{center}\includegraphics[width=0.9\textwidth]{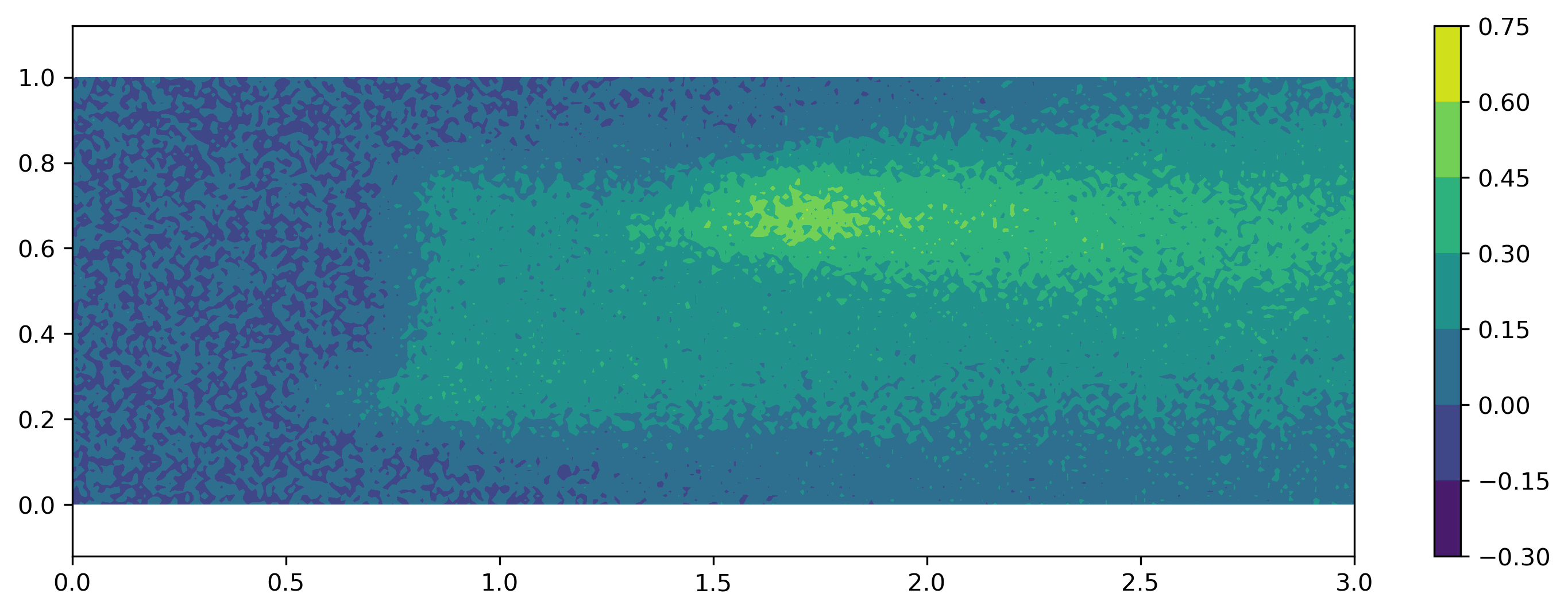}\end{center}
\caption{Advection-diffusion solution with normally distributed, additive noise according to $\mathcal{N}(0,\, u_\mathrm{max}\cdot 0.05)$.}
\label{fig:2dmeasurements}
\end{figure}

The underlying model is based on a diffusivity constant $c=0.01$ and a constant velocity field $\velocity = (1,0)^T$.
In order to recover the shapes we start the primal-dual active set algorithm with a regularization parameter $\mu = 5\mathrm{e}{-2}$ and accordingly $\gamma = 2\mathrm{e}{+1}$.
The effect of $\mu$ on the initial guess is as follows.
Decreasing $\mu$ leads to an initial solution that is closer to the original shape depicted Figure \ref{fig:2dsource}.
However there are unintended small shapes appearing all over the domain due to the noisy data.
Increasing the $\mu$ leads to an underestimation of the original shapes.
Yet, this initial guess is more smooth, which is visualized in Figure \ref{fig:2dcombined_levelset} as the dashed line.
\begin{figure}
\begin{center}\includegraphics[width=0.9\textwidth]{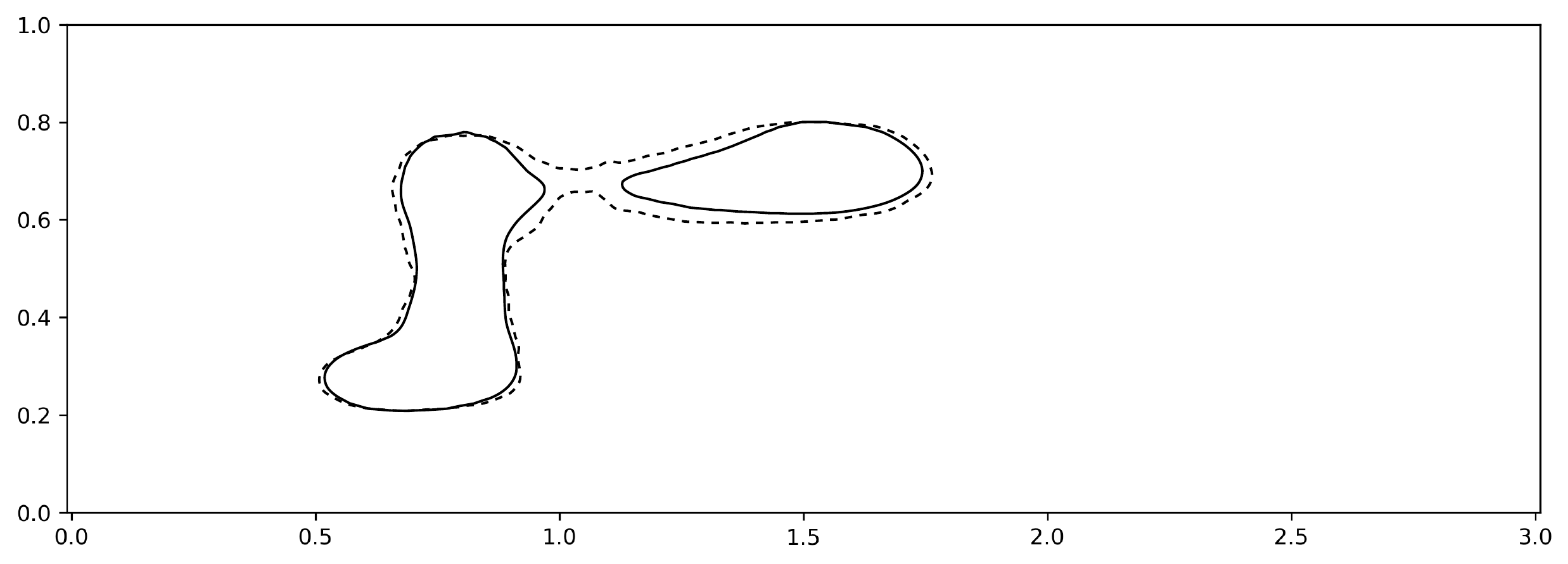}\end{center}
\caption{Initial level-set after relaxation problem and before shape optimization (dashed) and final level-set after shape optimization (solid).}
\label{fig:2dcombined_levelset}
\end{figure}
The strategy we follow is to find the smallest $\mu$ such that the number of shapes is stable and does not change, which is done manually for the moment.

Both the tolerance for the Newton iteration in line 3 of Algorithm \ref{alg1} and the tolerance for the inner MINRES iteration in line 5 are chosen to be $\epsilon_1 = 1\mathrm{e}{-12}$.
Figure \ref{fig:2diterations} visualized the number of MINRES iterations and the norm of the Newton residual.
\begin{figure}
\begin{center}\includegraphics[width=0.9\textwidth]{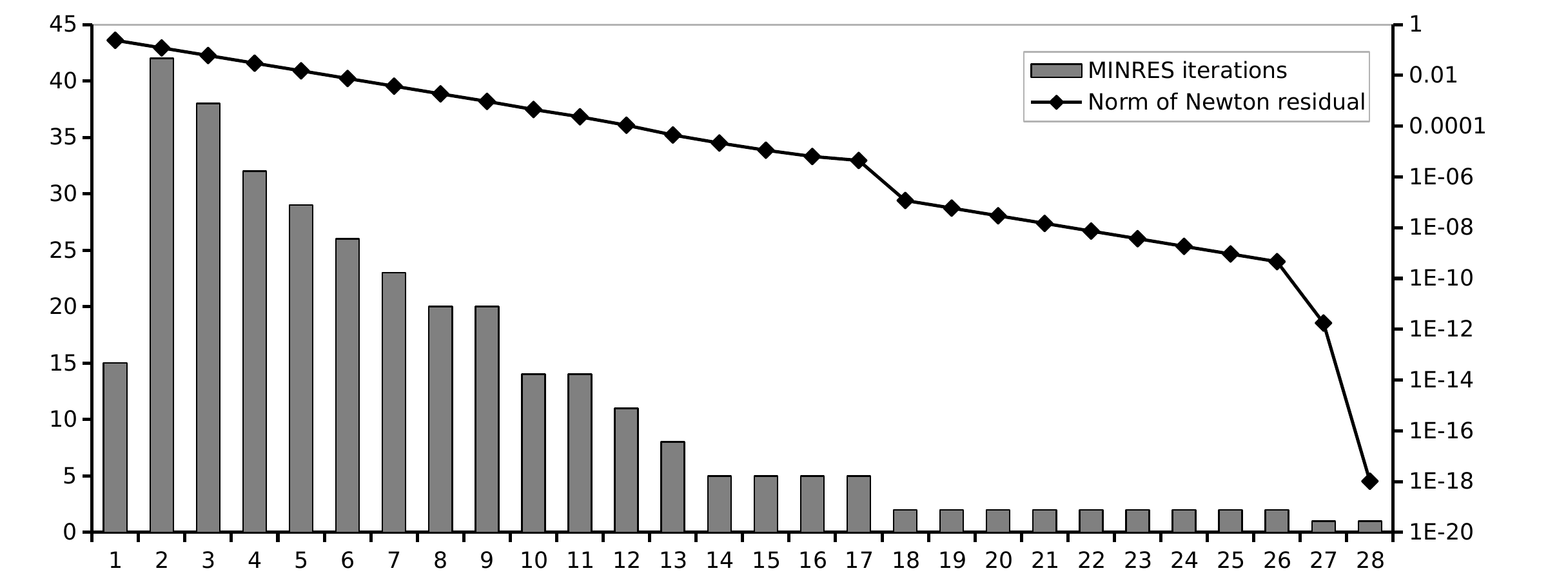}\end{center}
\caption{Number of MINRES iterations in each linear sub-problem in primal-dual active set algorithm together with l2-norm of Newton residual (log-scale).}
\label{fig:2diterations}
\end{figure}
In this example we only consider one computational grid for both the initial phase and the shape optimization.
Although it can be proven, that the outer Newton iteration has a mesh-independent convergence (cf.~\cite{hintermuller2002primal}), this is not true for the linear subproblems.
Note that in each MINRES iteration two PDEs have to be solved resulting in significant computational cost.
This issue is addressed later for the three dimensional case.

Figure \ref{fig:2dinitial_glyph} visualizes the initial shape and the first shape gradient.
\begin{figure}
\begin{center}\includegraphics[width=0.9\textwidth]{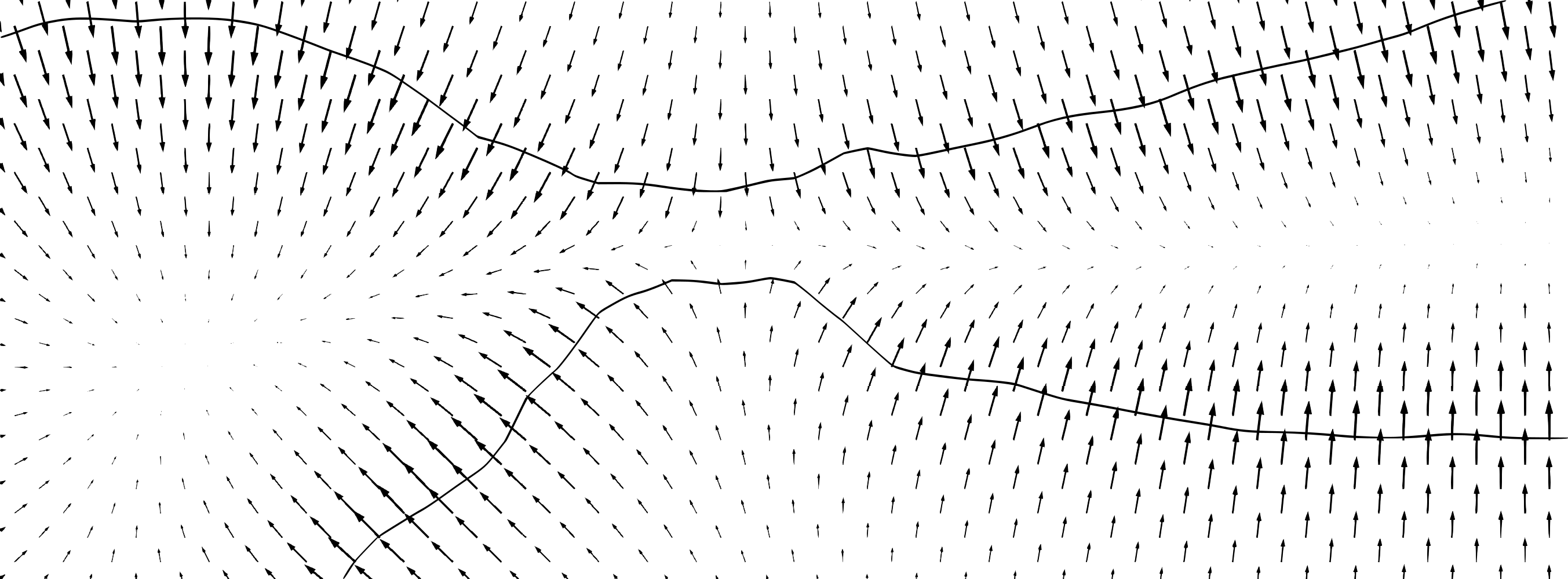}\end{center}
\caption{Zoom into the junction of the two shapes in Figure \ref{fig:2dcombined_levelset} with a visualization of the shape gradient and its impact to the level-set function.}
\label{fig:2dinitial_glyph}
\end{figure}
As pointed out in Section \ref{sec:algorithm} the advantage of a volume formulation of shape derivative $dJ(u,w,\Omega)[v]$ and gradient $\psi$ is, that $\partial \Oint$ does not have to be constructed geometrically at any time in the algorithm.
Only for the integration of the piecewise constant source term $f$, which arises in the state equation \eqref{eq:state} and the shape derivative \eqref{eq:shape_derivative}, the actual shape $\Oint$ has to be known.
The quadrature is performed according to Algorithm \ref{alg2}.
\begin{algorithm}
\caption{Choosing the quadrature order}
\label{alg2}
\begin{algorithmic}[1]
\itemsep0.05cm
\REQUIRE Let $\lbrace f_1, \dots, f_n\rbrace$ be the nodal values within one finite element $e_j$
\IF{$\min\lbrace f_1, \dots, f_n\rbrace \cdot \max\lbrace f_1, \dots, f_n\rbrace < 0$}
	\STATE Choose Gauss-Legendre nodes of order $q=2$
\ELSE
	\STATE Choose Gauss-Legendre nodes of order $q=8$
\ENDIF
\end{algorithmic}
\end{algorithm}
The idea is to detect cells where the level-set function $\phi$ passes zero values and increase the order of the underlying quadrature rule.
For our particular choice of $V_h$ these are cells where the nodal values in one element do not have the same sign.
Note that this procedure works for piecewise linear, bilinear or trilinear basis functions in $V_h$ only.
This quadrature is not exact, since the integrand is not a polynomial in the cells, where the level-set function passes zero.
Yet, this saves us from reconstructing a geometric representation of the set $\lbrace x \in \Omega : \phi(x) = 0 \rbrace$.
A more accurate but computationally more expensive strategy would for instance be to use techniques like extended finite elements (XFEM) \cite{dolbow1999finite}.

Figure \ref{fig:2dinitial_levelset} and \ref{fig:2dlevelset_function} show the transport of the level-set.
\begin{figure}
\begin{center}\includegraphics[width=0.9\textwidth]{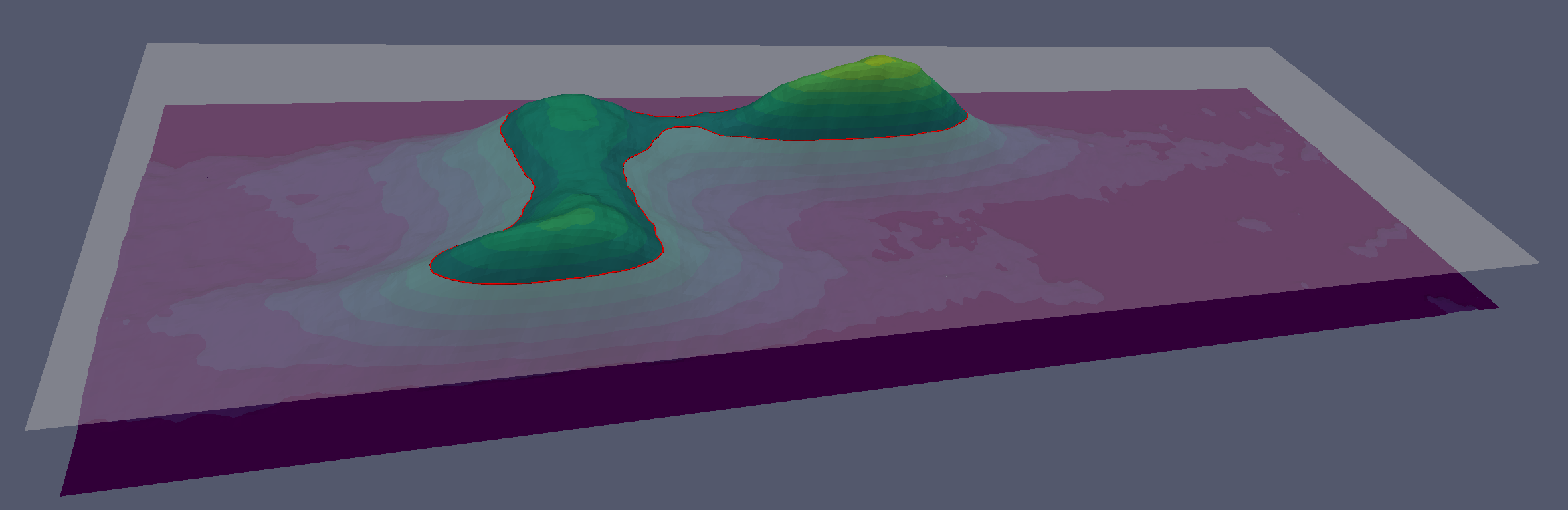}\end{center}
\caption{Initial level-set after relaxation problem and before shape optimization.}
\label{fig:2dinitial_levelset}
\end{figure}
\begin{figure}
\begin{center}\includegraphics[width=0.9\textwidth]{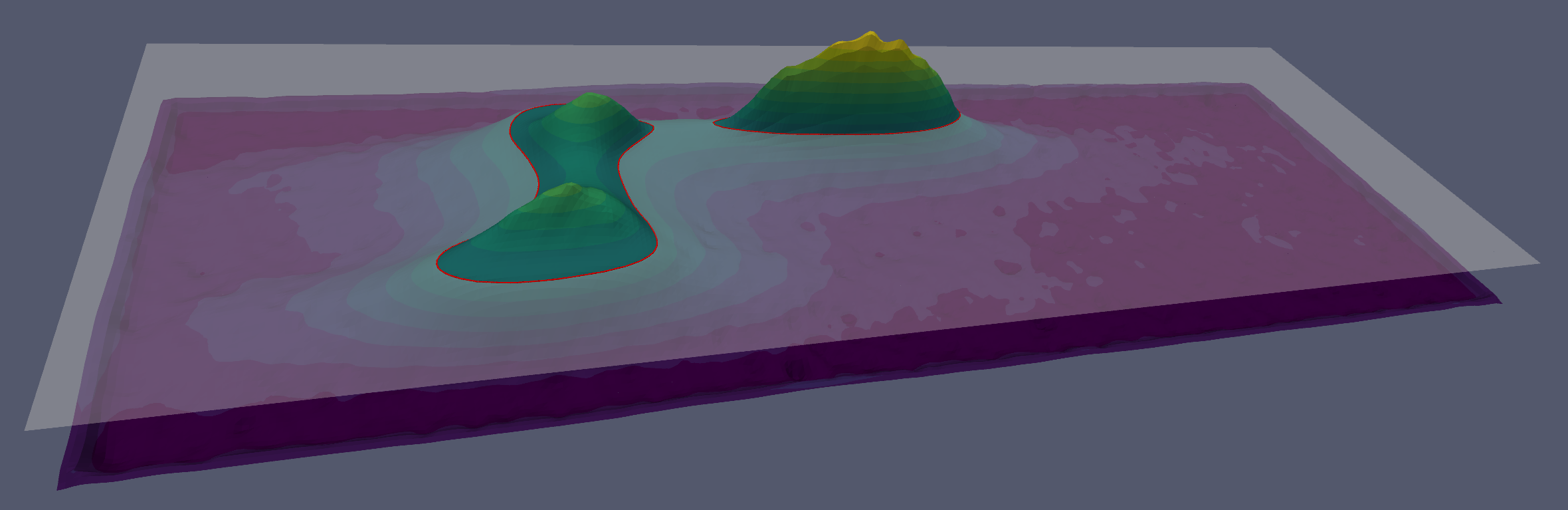}\end{center}
\caption{Final level-set function after shape optimization.}
\label{fig:2dlevelset_function}
\end{figure}
In the first figure $\phi^0$ is visualized as the height of the surface.
The transparent plane represents the zero level and the red line, which is the intersection of this plane and the level-set function, is interpreted as $\partial \Oint$.
Figure \ref{fig:2dlevelset_function} shows the same final level-set after the shape optimization part of Algorithm \ref{alg1}.

For the representation of the gradient $\psi$ we choose $\alpha = 1\mathrm{e}{-2}$.
The transport of the level-set function $\phi$ given by the model \eqref{eq:time_dependent_adv_diff} is discretized using a backward Euler time stepping with a step-length of $\Delta t = 1$.
With each gradient $\phi$ one time step is computed.
The values of the gradient $\phi$ are decreasing during the optimization, which makes it necessary to link the diffusion coefficient $\epsilon$ to $\phi$.
Otherwise, the transport of the level-set would turn into diffusion dominated model, which is not intended.
Let $\phi_\mathrm{max}$ be the maximum of the absolute values of the discretized gradient $\phi$.
We then choose $\epsilon = \phi_\mathrm{max} \cdot 2\mathrm{e}{-3}$.

The objective function during the second phase of Algorithm \ref{alg1} can be seen in Figure \ref{fig:2dobjective}.
\begin{figure}
\begin{center}\includegraphics[width=0.9\textwidth]{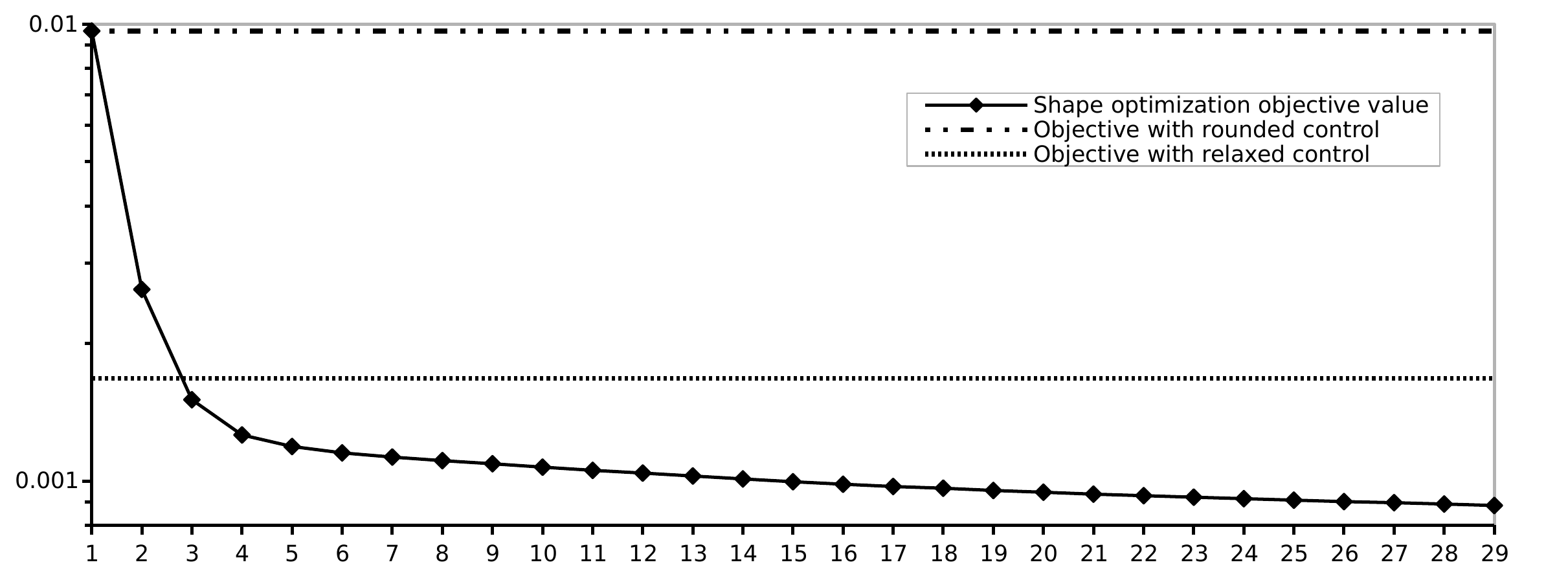}\end{center}
\caption{Objective value during shape optimization and level of objective function with respect to relaxed control (dashed line, log-scale).}
\label{fig:2dobjective}
\end{figure}
Here we can see the difference between the model and the measurements $\Vert u - \bar{u}\Vert_{L^2(\Omega)}$ for the relaxed control (dashed line) and after rounding the relaxed control with respect to the mean average (dotted line).
We observe, that the relaxed control leads to a better objective value than the rounded one.
Yet, by moving the level set slightly in the direction of the shape gradient we can even outperform the relaxed solution.
The algorithm stops according to the criterion in line 12 for $\epsilon_2 = 1\mathrm{e}{-4}$.

\begin{figure}
\begin{center}
\includegraphics[width=0.9\textwidth]{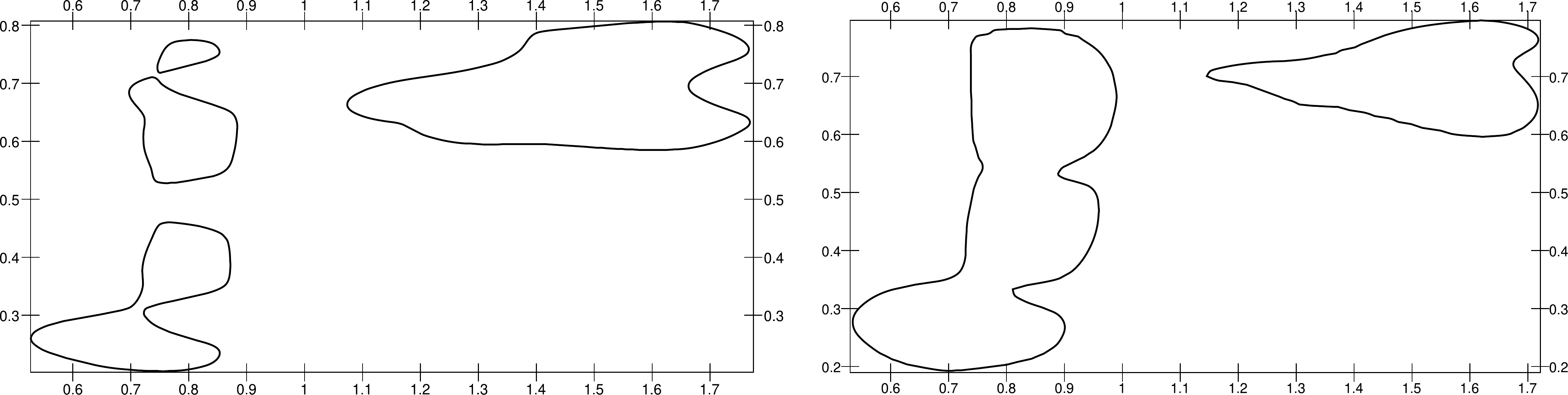}
\end{center}
\caption{Inital guess (left) and final solution after 60 shape optimization iterations (right) with partial measurements}
\label{fig:partial_measurements}
\end{figure}
We now run the same test as before, yet with partial measurements only.
The setting is similar to the one in Figure \ref{fig:domain}.
We have a grid of $6\times 6$ sensors, which are equidistantly distributed, such that $\Om$ covers $26.4\%$ of $\Omega$.
The purpose of this test is, that the relaxation problem overestimates the number of shapes as depicted in Figure \ref{fig:partial_measurements} on the left.
Here it can be seen on the right, that after the stopping criterion is reached, three shapes are united into one.
These results are achieved by increasing the diffusion in the level-set transport equation to $\epsilon = \phi_\mathrm{max} \cdot 5\mathrm{e}{-3}$.

Our second example is a three dimensional case similar to the two dimensional one.
Here we want to point out the computational cost for the relaxation problem.
Since we only want a rough estimation of the number and position of the shapes $\Oint$, we solve this problem on a coarser level.
As pointed out earlier we solve the state \eqref{eq:state} and adjoint equation \eqref{eq:adjoint} with an iterative solver.
Here we choose a GMRES iteration preconditioned with geometric multigrid, which is based on a hierarchical grid structure.
The standard geometric multigrid preconditioner of the PETSc toolbox is used with one V-cycle per GMRES iteration.
The preconditioner is configured with two SOR pre-smoothing and two post-smoothing steps, respectively.
We thus have the geometric structure and interpolation operators of several finite element spaces $V_{h_1} \subset \dots \subset V_{h_l}$ available.
The spaces $V_{h_i}$ are again based on linear finite elements.
This is utilized in line 10 of Algorithm \ref{alg1}, where the initial guess obtained by the relaxation solver can optionally be interpolated to a fine grid for the shape optimization.

\begin{figure}
\begin{minipage}{0.49\textwidth}
\begin{center}\includegraphics[width=1.0\textwidth]{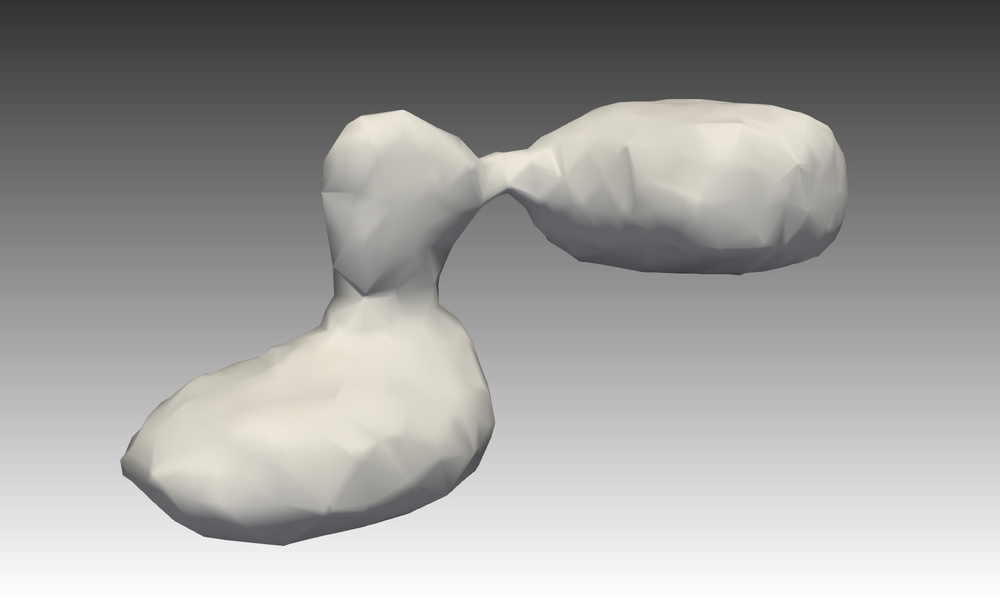}\end{center}
\begin{center}\includegraphics[width=1.0\textwidth]{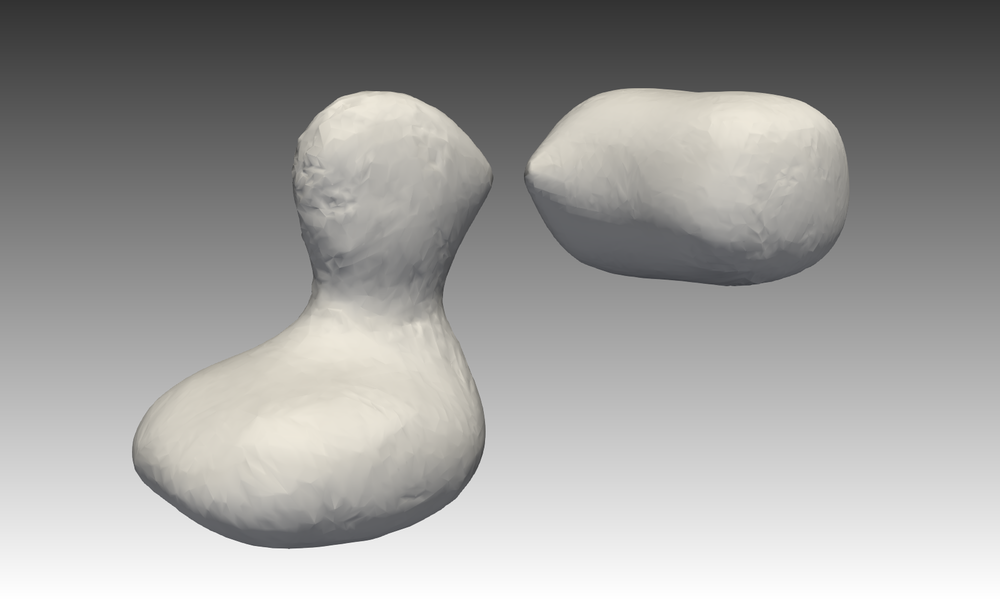}\end{center}
\end{minipage}
\begin{minipage}{0.49\textwidth}
\begin{center}\includegraphics[width=1.0\textwidth]{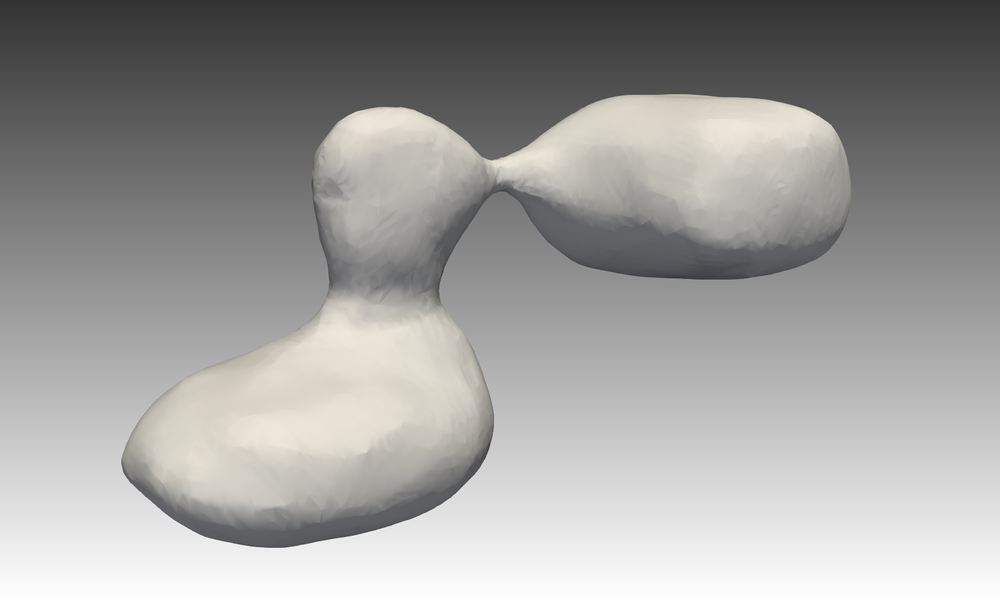}\end{center}
\begin{center}\includegraphics[width=1.0\textwidth]{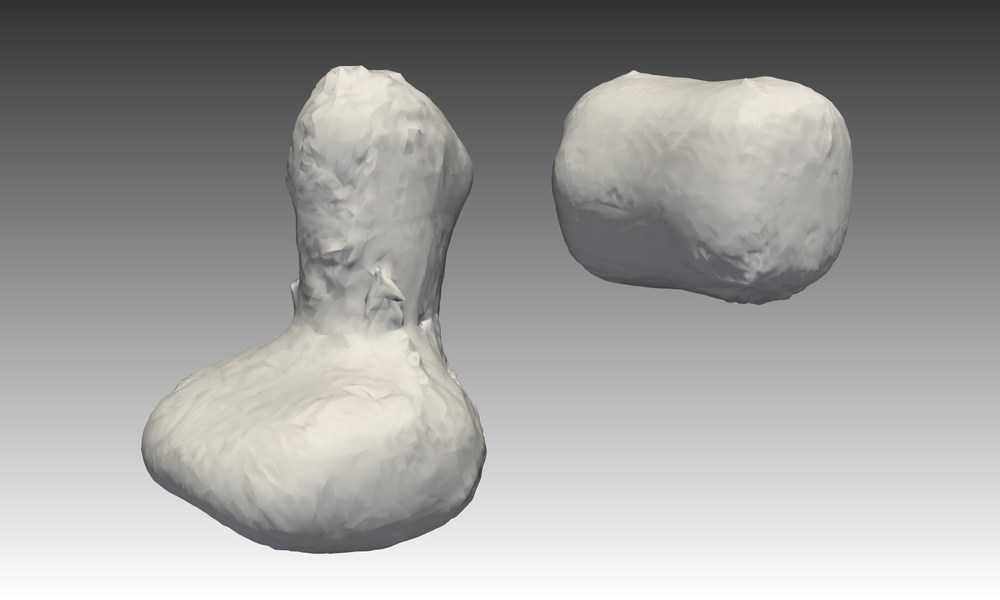}\end{center}
\end{minipage}
\caption{Initial configuration after relaxation problem on coarse grid (top left) and level-set after 10, 50 and 100 (bottom right) shape optimization steps.}
\label{fig:3dresults}
\end{figure}
The computational domain for the three dimensional test is a cuboid given by $ \Omega = \left[ 0, 2 \right] \times \left[ 0, 1 \right] \times \left[ 0, 1 \right]$.
Further, the source term, which is used to generate the measurements, is chosen to be $\Oint = \left[ 0.3, 0.7 \right] \times \left[ 0.4, 0.8 \right] \times \left[ 0.2, 0.4 \right] \,\cup\, \left[ 0.5, 0.7 \right] \times \left[ 0.5, 0.7 \right] \times \left[ 0.2, 0.8 \right] \,\cup\, \left[ 0.9, 1.3 \right] \times \left[ 0.2, 0.6 \right] \times \left[ 0.5, 0.8 \right]$.
The velocity field is given by the constant $\velocity = (1,\, 0,\, 0)^T$.
We have again additive noise on the measurements, which is distributed according to $\mathcal{N}(0,\, u_\mathrm{max}\cdot 0.05)$.
These measurements are generated on fine grid and the multigrid restriction is used to obtain a representation on the coarse grid for the relaxation problem.
$\Omega$ is discretized by $10,350$ tetrahedrons on the coarse grid using the Delaunay algorithm in the mesh generator GMSH \cite{geuzaine2009gmsh}.
We then obtain by three hierarchical refinements $5,299,200$ elements on the finest grid level.
The second grid $V_H \coloneqq V_{h_2}$ is used for the relaxation problem and the finest grid $V_H \coloneqq V_{h_4}$ for the shape optimization.

Figure \ref{fig:3dresults} shows the $\partial \Oint$ after the relaxation problem and after 10, 50, 100 iterations of the shape optimization.
In the first picture it can be seen how the intial guess is influenced by edges in the coarse grid.
This is smoothed out in the following shape optimization iterations on the much finer grid.
The parameter are chosen similar to the two dimensional case as $\mu = 5\mathrm{e}{-2}$, $\gamma = 2\mathrm{e}{+1}$, $\alpha = 1\mathrm{e}{-1}$, $\epsilon_1 = 1\mathrm{e}{-12}$ and $\epsilon_2 = 1\mathrm{e}{-4}$.
The propagation of the level-set function is controlled by $\epsilon = \phi_\mathrm{max} \cdot 6\mathrm{e}{-3}$, $\Delta t$ and one time step per optimization iteration.
As a result we can see that again it is possible to split the shape into two.

\section{Conclusion}
\label{sec:conclusion}
This paper presents a method towards mixed integer PDE constrained optimization by circumventing the computational complexity of combinatorial optimization approaches.
The original PDE constrained problem with binary variables, which represent a control that is either active or inactive, is replaced by a relaxed problem leading to an approximate solution.
Yet, this approximate solution is a control with continuous and not discrete values.
It turns out that a rounding strategy based on the average value of the relaxed control may lead to poor results.
Using this as an initial guess we improve the control by interpreting it as a level-set function and moving it slightly in the direction of a shape gradient.
From a computational point of view this method is attractive since it can be formulated for structured grids.
The important fact is, that the interface, which separates active and inactive control, does not have to be constructed during the algorithm.

We also demonstrate that the relaxation problem can be computed on a much finer grid than the shape problem leading to a computationally cheap algorithm with high spatial resolution.
Moreover, we are able to bypass some major challenges of shape optimization methods.
Since the shape part of our algorithm starts with a good initial guess, large deformations are not necessary.
We also demonstrate that we are not bound to the topology of the initial guess.
The domain formulation of shape derivatives and gradients together with a small amount of diffusion on the level-set transport equation enables us to join and split shapes.
\section*{Acknowledgment}
This work has been supported by the German Research Foundation (DFG) within the priority program SPP 1648 ``Software for Exascale Computing'' under contract number Schu804/12-1 and the research training group 2126 ``Algorithmic Optimization''.

\bibliographystyle{plain}
\bibliography{citations.bib}

\begin{thebibliography}{10}

\bibitem{allaire2012shape}
G.~Allaire.
\newblock {\em Shape optimization by the homogenization method}, volume 146.
\newblock Springer Science \& Business Media, 2012.

\bibitem{allaire2004structural}
G.~Allaire, F.~Jouve, and A.M. Toader.
\newblock Structural optimization using sensitivity analysis and a level-set
  method.
\newblock {\em Journal of computational physics}, 194(1):363--393, 2004.

\bibitem{bergounioux1997augemented}
M.~Bergounioux and K.~Kunisch.
\newblock Augemented {L}agrangian techniques for elliptic state constrained
  optimal control problems.
\newblock {\em SIAM Journal on Control and Optimization}, 35(5):1524--1543,
  1997.

\bibitem{cheney1999electrical}
M.~Cheney, D.~Isaacson, and J.~Newell.
\newblock Electrical impedance tomography.
\newblock {\em SIAM review}, 41(1):85--101, 1999.

\bibitem{correa1985directional}
R.~Correa and A.~Seeger.
\newblock {Directional derivative of a minmax function}.
\newblock {\em Nonlinear Analysis: Theory, Methods \& Applications},
  9(1):13--22, 1985.

\bibitem{delfour2001shapes}
M.C. Delfour and J.-P. Zol\'esio.
\newblock {\em {Shapes and Geometries: Metrics, Analysis, Differential
  Calculus, and Optimization}}, volume~22 of {\em Advances in Design and
  Control}.
\newblock SIAM, 2nd edition, 2001.

\bibitem{evans1993partial}
L.C. Evans.
\newblock {\em {Partial Differential Equations}}.
\newblock American Mathematical Society, 1993.

\bibitem{gangl2015shape}
P.~Gangl, U.~Langer, A.~Laurain, H.~Meftahi, and K.~Sturm.
\newblock Shape optimization of an electric motor subject to nonlinear
  magnetostatics.
\newblock {\em SIAM Journal on Scientific Computing}, 37(6):B1002--B1025, 2015.

\bibitem{geuzaine2009gmsh}
C.~Geuzaine and J.-F. Remacle.
\newblock Gmsh: A 3-d finite element mesh generator with built-in pre-and
  post-processing facilities.
\newblock {\em International journal for numerical methods in engineering},
  79(11):1309--1331, 2009.

\bibitem{giles2000introduction}
M.~B. Giles and N.~A. Pierce.
\newblock An introduction to the adjoint approach to design.
\newblock {\em Flow, turbulence and combustion}, 65(3-4):393--415, 2000.

\bibitem{hante2013relaxation}
F.~M. Hante and S.~Sager.
\newblock Relaxation methods for mixed-integer optimal control of partial
  differential equations.
\newblock {\em Computational Optimization and Applications}, 55(1):197--225,
  2013.

\bibitem{haslinger2003advancecs}
J.~Haslinger and R.~M\"{a}kinen.
\newblock {\em {Introduction to Shape Optimization: Theory, Approximation, and
  Computation}}, volume~7 of {\em Advances in Design and Control}.
\newblock SIAM, 2003.

\bibitem{herzog2010algorithms}
R.~Herzog and K.~Kunisch.
\newblock Algorithms for {PDE}-constrained optimization.
\newblock {\em GAMM-Mitteilungen}, 33(2):163--176, 2010.

\bibitem{hintermuller2002primal}
M.~Hinterm{\"u}ller, K.~Ito, and K.~Kunisch.
\newblock The primal-dual active set strategy as a semismooth newton method.
\newblock {\em SIAM Journal on Optimization}, 13(3):865--888, 2002.

\bibitem{hintermuller2008electrical}
M.~Hinterm{\"u}ller and A.~Laurain.
\newblock Electrical impedance tomography: from topology to shape.
\newblock {\em Control \& Cybernetics}, 37(4), 2008.

\bibitem{jameson2003aerodynamic}
A.~Jameson.
\newblock Aerodynamic shape optimization using the adjoint method.
\newblock {\em Lectures at the Von Karman Institute, Brussels}, 2003.

\bibitem{laurain2016distributed}
A.~Laurain and K.~Sturm.
\newblock Distributed shape derivative via averaged adjoint method and
  applications.
\newblock {\em ESAIM Math. Model. Numer. Anal.}, 50(4):1241--1267, 2016.

\bibitem{dolbow1999finite}
N.~Mo{\"e}s, J.~Dolbow, and T.~Belytschko.
\newblock A finite element method for crack growth without remeshing.
\newblock {\em International journal for numerical methods in engineering},
  46(1):131--150, 1999.

\bibitem{mohammadi2001applied}
B.~Mohammadi and O.~Pironneau.
\newblock {\em {Applied Shape Optimization for Fluids}}.
\newblock Oxford University Press, 2001.

\bibitem{nittka2011regularity}
R.~Nittka.
\newblock Regularity of solutions of linear second order elliptic and parabolic
  boundary value problems on lipschitz domains.
\newblock {\em Journal of Differential Equations}, 251(4-5):860--880, 2011.

\bibitem{berggren2016largescale}
S.~Schmidt, E.~Wadbro, and M.~Berggren.
\newblock Large-scale three-dimensional acoustic horn optimization.
\newblock {\em SIAM Journal on Scientific Computing}, 38(6):B917--B940, 2016.

\bibitem{schulz2016computational}
V.~Schulz and M.~Siebenborn.
\newblock Computational comparison of surface metrics for {PDE} constrained
  shape optimization.
\newblock {\em Computational Methods in Applied Mathematics}, 2016.

\bibitem{schulz2016efficient}
V.~Schulz, M.~Siebenborn, and K.~Welker.
\newblock Efficient {PDE} constrained shape optimization based on
  {S}teklov--{P}oincar\'e-type metrics.
\newblock {\em SIAM Journal on Optimization}, 26(4):2800--2819, 2016.

\bibitem{sokolowski1992introduction}
J.~Sokolowski and J.-P. Zolesio.
\newblock Introduction to shape optimization.
\newblock In {\em Introduction to Shape Optimization}, pages 5--12. Springer,
  1992.

\bibitem{troeltzsch2010optimal}
F.~Tröltzsch.
\newblock Optimal control of partial differential equations.
\newblock {\em Graduate Studies in Mathematics}, 2010.

\bibitem{ulbrich2002semismooth}
M.~Ulbrich.
\newblock Semismooth newton methods for operator equations in function spaces.
\newblock {\em SIAM Journal on Optimization}, 13(3):805--841, 2002.

\bibitem{weiser2008control}
M.~Weiser, T.~G{\"a}nzler, and A.~Schiela.
\newblock A control reduced primal interior point method for a class of control
  constrained optimal control problems.
\newblock {\em Computational Optimization and Applications}, 41(1):127--145,
  Sep 2008.

\bibitem{welker2016efficient}
K.~Welker.
\newblock {\em {Efficient PDE Constrained Shape Optimization in Shape Spaces}}.
\newblock PhD thesis, Universit\"{a}t Trier, 2016.

\end{thebibliography}

\end{document}